\documentclass[hidelinks,12pt]{amsart}
\usepackage{mathrsfs}
\usepackage{amsfonts, amssymb, amsmath, amscd}
\usepackage{enumerate}
\usepackage{graphicx}
\usepackage{tikz}
\usepackage{pgfplots}
\usepackage{float}
\pgfplotsset{compat=1.14}
\usetikzlibrary{shapes}
\usetikzlibrary{plotmarks}
\usepackage{stmaryrd}
\usepackage{booktabs}
\usepackage{mathtools}
\usepackage{young}
\usepackage{color,soul}
\usepackage{MnSymbol,wasysym}
\usepackage{calligra,mathrsfs}
\usepackage{ulem}
\usetikzlibrary{matrix,arrows}

\usepackage[left=.9truein,right=.9truein,bottom=.9truein,top=.9truein]{geometry}

\makeatletter
\def\@tocline#1#2#3#4#5#6#7{\relax
  \ifnum #1>\c@tocdepth 
  \else
    \par \addpenalty\@secpenalty\addvspace{#2}%
    \begingroup \hyphenpenalty\@M
    \@ifempty{#4}{%
      \@tempdima\csname r@tocindent\number#1\endcsname\relax
    }{%
      \@tempdima#4\relax
    }%
    \parindent\z@ \leftskip#3\relax \advance\leftskip\@tempdima\relax
    \rightskip\@pnumwidth plus1em \parfillskip-\@pnumwidth
    #5\leavevmode\hskip-\@tempdima #6\relax
    \dotfill\hbox to\@pnumwidth{\@tocpagenum{#7}}\par
    \nobreak
    \endgroup
  \fi}
\makeatother

\let\oldtocsection=\tocsection
\let\oldtocsubsection=\tocsubsection

\renewcommand{\tocsection}[2]{\hspace{0em}\oldtocsection{#1}{#2}}
\renewcommand{\tocsubsection}[2]{\hspace{1.25em}\oldtocsubsection{#1}{#2}}

\usepackage{hyperref}
\hypersetup{pdfborder = {0 0 0}}

\newcommand{\sheafhom}{\mathscr{H}\text{\kern -3pt {\calligra\large om}}\,}
\newcommand{\sheafext}{\mathscr{E}\text{\kern -3pt {\calligra\large xt}}\,}

\newcommand{\Z}{\mathbb{Z}}
\newcommand{\bcup}{\bigcup}
\newcommand{\bcap}{\bigcap}
\newcommand{\R}{\mathbb{R}}

\newtheorem{theo}{Theorem}[section]
\newtheorem{theorem}[theo]{Theorem}
\newtheorem*{theorem*}{Theorem}
\newtheorem{thm}[theo]{Theorem}
\newtheorem*{thm*}{Theorem}
\newtheorem{proposition}[theo]{Proposition}
\newtheorem*{proposition*}{Proposition}
\newtheorem{prop}[theo]{Proposition}
\newtheorem*{prop*}{Proposition}
\newtheorem{remark}[theo]{Remark}
\newtheorem*{remark*}{Remark}
\newtheorem{lemma}[theo]{Lemma}
\newtheorem*{lemma*}{Lemma}

\newtheorem*{cor*}{Corollary}

\newtheorem*{claim*}{Claim}

\newtheorem*{details*}{Details}

\newtheorem*{recall*}{Recall}

\newtheorem*{ass*}{Assumption}

\newtheorem*{conj*}{Conjecture}

\newtheorem*{intprob*}{The Interpolation Problem}

\theoremstyle{definition}

\newtheorem*{definition*}{Definition}

\newtheorem*{deff*}{Definition}
\newtheorem{exmp}{Example}[section]

\newtheorem*{problem*}{Problem}

\newtheorem*{prob*}{Problem}

\begin{document}
\title{The sum of the Betti numbers of smooth Hilbert schemes}
\author[Donato]{Joseph Donato}
\email{jsdonato@umich.edu}
\author[Lewis]{Monica Lewis}
\email{malewi@umich.edu}
\author[Ryan]{Tim Ryan}
\email{rtimothy@umich.edu}
\author[Udrenas]{Faustas Udrenas}
\email{fudrenas@umich.edu}
\author[Zhang]{Zijian Zhang}
\email{zzjharry@umich.edu}

\address{Department of Mathematics, University of Michigan, Ann Arbor, MI}

\begin{abstract}
In this paper, we compute the sum of the Betti numbers for 6 of the 7 families of smooth Hilbert schemes over projective space.
\end{abstract}
\maketitle
\setcounter{tocdepth}{2}

\section{Introduction}
\noindent 
Hilbert schemes are one of the classic families of varieties. 
In particular, Hilbert schemes of points on surfaces have been extensively studied; so extensively studied that any reasonable list of example literature would take several pages, see \cite{N,N2,G} for introductions to the area.
This study was at least in part due to these being one of the only sets of Hilbert schemes which were known to be smooth.
Recent work \cite{SS} has characterized exactly which Hilbert schemes on projective spaces are smooth by giving seven ``families'' of smooth Hilbert schemes.

\medskip \noindent
One of the most fundamental topological properties of an algebraic variety is its homology.
The homology of Hilbert schemes of points has been extensively studied, e.g. \cite{ES,G2,Gr,LS,LQW,E}.
It is an immediate consequence of \cite{BB} that the smooth Hilbert schemes over $\mathbb{C}$ have freely generated even homology groups and zero odd homology groups.
This was used to compute the Betti numbers of Hilbert schemes of points on the plane in \cite{ES}.
A natural follow up question then is what are the ranks of the homology groups for all of the smooth Hilbert schemes?
In this paper, we compute the sum of the Betti numbers for six of the seven families of smooth Hilbert schemes.
Since these Hilbert schemes are smooth, this is equivalent to computing the dimension of the cohomology ring as a vector space over $\mathbb{C}$.
Note in this case the cohomology and the Chow rings are isomorphic.

\medskip \noindent
In order to state the theorem, recall that Macaulay proved that the Hilbert scheme of subschemes of $\mathbb{P}^n$ with Hilbert polynomial $p$, denoted $\mathbb{P}^{n[p]}$, is nonempty if and only if $p$ can be written in the form $p(t) = \sum_{i=1}^r \binom{t+\lambda_i -i}{\lambda_i-1}$ for some integer partition $\lambda = (\lambda_1,\cdots,\lambda_r)$ of integers satisfying $\lambda_1\geq \cdots \geq \lambda_r \geq 1$.
\begin{theorem}
\label{thm: main}
Let $H_{n,\lambda}$ be the sum of the Betti numbers for $\mathbb{P}^{n[p_\lambda]}$ where $p_\lambda$ corresponds to the integer partition $\lambda = (\lambda_1,\dots,\lambda_r)$. \\
(1.1) If $n=1$ and $\lambda=(1^r)$, then $H_{n,\lambda} = r+1$.\\
(1.2) If $n=2$ and  $\lambda=(2^m,1^r)$\footnote{Note, we write $a^b$ in place of repeating $a$ in the partition $b$ times for convenience}, then $H_{n,\lambda} = \binom{m+2}{2}\cdot \sum_{c_1+c_2+c_3=r}\Big[f_1(c_1)\cdot f_1(c_2)\cdot f_1(c_3)\Big]$ where $f_1$ is the integer partition function.\\
(3) If $\lambda = (1)$ or $\lambda = (n^{r-2},\lambda_{r-1},1)$ where $r\geq 2$ and $n\geq \lambda_{r-1} \geq 1$, then $H_{n,\lambda} = \binom{n+r-2}{r-2}\binom{n+1}{\lambda_{r-1}}(n+1)$.\\
(4) If $\lambda = (n^{r-s-3},\lambda_{r-s-2}^{s+2},1)$ where $r-3\geq s \geq 0$ and $n-1\geq \lambda_{r-s-2}\geq 3$, then \par $H_{n,\lambda} = \binom{n+r-s-3}{r-s-3}\left(\binom{n+1}{\lambda_{r-s-2}+1}\left(\binom{\lambda_{r-s-2}+1+s+2}{s+2}-(\lambda_{r-s-2}+1)\right)+\binom{n+1}{\lambda_{r-s-2}}\right)(n+1)$.\\
(5) If $\lambda = (n^{r-s-5},2^{s+4},1)$ where $r-5 \geq s \geq 0$, then \par $H_{n,\lambda} = \binom{n+r-s-5}{r-s-5}\left(\binom{n+1}{3}\left(\binom{3+s+4}{s+4}-3\right)+\binom{n+1}{2}\right)(n+1)$. \\
(6) If $\lambda = (n^{r-3},1^3)$ where $r\geq 3$, then $H_{n,\lambda} = \binom{k+n}{n}\frac{n(n+1)(5n+1)}{3}$. \\
(7) If $\lambda = (m+1)$ or $r=0$, then $H_{n,\lambda} = 1$. \\
\end{theorem}

\vspace{-.18in}
\noindent Note, the numbering is nonstandard so as to match the numbering in the families of smooth Hilbert schemes in \cite{SS}.

\medskip\noindent
The proof of this results works by translating computing the ranks of the homology groups into counting saturated monomial ideals and then translating that into counting choices of orthants in an $(n+1)$-dimensional lattice.
The cases are proved in proven in Propositions \ref{prop: line}, \ref{prop: 1}, \ref{prop: 6}, \ref{prop: 3}, \ref{prop: 5}, and \ref{prop: 4}.

\medskip \noindent
The organization of the paper is as follows.
In Section \ref{sec: background}, the necessary background is given.
In Section \ref{sec: line}, the case of Hilbert schemes over the projective line is worked out as an example case.
In Section \ref{sec: plane}, the case of Hilbert schemes over the projective plane is worked out, which proves case (1.1) of Theorem \ref{thm: main}.
Finally, in Section \ref{sec: general}, we prove cases (3)-(7) of Theorem \ref{thm: main}.

\medskip \noindent
The authors would like to thank Harry Bray, who directed the Laboratory of Geometry at Michigan where this work was carried out, and Zhan Jiang, for many enlightening conversations.

\section{Background}
\label{sec: background}
\noindent
In this section, we review the necessary background material.

\medskip\noindent
Let $I \subset \mathbb{C}[x_0,\dots,x_n]$ be a homogeneous ideal.
As the quotient ring $R/I$ is a graded ring, it comes equipped with a Hilbert function, $h_I: \mathbb{Z}_{\geq 0} \to \mathbb{Z}_{\geq 0}$ which sends $d$ to the dimension of the degree $d$ graded piece.
By Hilbert \cite{H}, this function agrees with a polynomial $H_I$ for $d>>0$.
This is the \textit{Hilbert polynomial} of $I$; note this is more properly the Hilbert polynomial of $R/I$, but no confusion will arise by this usage.

\medskip\noindent
The degree of this polynomial is the dimension of $V(I)$, and the other coefficients include other geometric information such as the degree.
As the polynomial captures a lot of the geometry of the subvariety cut out by $I$, a natural definition of equivalence on algebraic subvarieties of $\mathbb{P}^n$ is those with the same Hilbert polynomial.
By Grothendieck \cite{Gro}, the set of subvarieties of $\mathbb{P}^n$ with the same Hilbert polynomial $p$ forms an algebraic scheme called the \textit{Hilbert scheme}, denoted $\mathbb{P}^{n[p]}$.
The first question one can ask about these Hilbert schemes is when are they nonempty?
This was answered by Macaualay.

\begin{theorem}\cite{M} 
Given \(R = \mathbb{C}[x_0, \cdots, x_n]\) and polynomial \(p(d)\) in one variable, there exists ideals in R with Hilbert polynomial $p(d)$ if and only if \(p(d)\) can be written in the form \(\Sigma^{m}_{i = 1} \binom{d + \lambda_i - i}{\lambda_i - 1}\) for some integer partition \(n \geqslant \lambda_1 \geqslant \cdots \geqslant \lambda_m \geqslant 1\).
\end{theorem}

\medskip\noindent 
This theorem means any $\lambda$-sequence defines a Hilbert polynomial, and we can also get a $\lambda$-sequence from any Hilbert polynomial written in the form shown in Theorem 3.2. For example, if $\lambda = (3, 2)$, then the corresponding Hilbert polynomial is \(\binom{d + 3 - 1}{3 - 1} + \binom{d + 2 - 2}{2 - 1} = \binom{d + 2}{2} + \binom{d}{1} = \frac{(d + 2)(d + 1)}{2!} + d = \frac{1}{2}d^2 + \frac{5}{2}d + 1\).
We will abuse notation and refer interchangeably to $\lambda$ and $p_\lambda$.

\medskip\noindent
It is a natural question to ask about the homology of a Hilbert scheme, but this question is most interesting for the smooth Hilbert schemes where Poincare duality holds, which makes the homology dual to the cohomology.
That naturally leads one to ask which Hilbert schemes are smooth.
This was recently answered by Skjelnes and Smith \cite{SS}, Note, the $\lambda's$ in the following Theorem  are exactly the same as the $\lambda's$ mentioned in previous theorem.
\begin{theorem}\cite{SS}
\text Let p be a polynomial in a single variable with some sequence \(\lambda = (\lambda_1,..., \lambda_r)\) with \(n \geqslant \lambda_1 \geqslant ... \geqslant \lambda_r \geqslant 1\). Then the Hilbert scheme \(Hilb^p(\mathbb{P}^{n})\) on projective space is smooth if and only if:
\par 1. \(n \leqslant 2\), 
\par 2. \(\lambda_r \geqslant 2\), 
\par 3. \(r \leqslant 1\)\,or\,\(\lambda = (n^{r-2}, \lambda^{1}_{r-1}.1^{1})\)\,for all\,\,\(r \geqslant 2\), 
\par 4. \(\lambda = (n^{r - s - 3}, \lambda^{s + 2}_{r - s - 2}, 2^0, 1^1) \) for all \(r \geqslant 2\), 
\par 5. \(\lambda = (n^{r - s - 5}, 2^{s + 4}, 1^1)\) for all \(0 \leqslant s \leqslant r - 5\) and all \(r \geqslant 5\),
\par 6. \(\lambda = (n^{r - 3}, 1^3)\) for all \(r \geqslant 3\)\text{, or }
\par 7. \(r = 0 \text{ or }(n+1)\)
\end{theorem}

\medskip\noindent
The Hilbert schemes over $\mathbb{P}^n$ inherit the $\mathrm{PGL}(n+1)$ action from projective space itself,
In particular, this restricts to a $\mathbb{C}^*$-action; actually it restrict to many different $\mathbb{C}^*$-actions, any of which will work. 
The fixed points of this action are points corresponding to the finitely many saturated monomial ideals with that Hilbert polynomial.
That let's us apply the following theorem of Bialynicki-Birula for smooth Hilbert schemes.
\begin{thm}\cite{BB}
Let $X$ be a smooth projective varity with an action of $\mathbb{C}^*$. Suppose that the fixpoint set of $\{p_1,\cdots,p_m\}$ is finite and let $X_i = \{x\in X : \lim_{t\to 0} tx = p_i\}$. Then $X$ has a cellular decomposition with cells $X_i$.
\end{thm}

\medskip\noindent
Pairing that with the following result of Fulton shows that a smooth Hilbert scheme has freely generated even cohomology groups and no odd cohomology groups.

\begin{thm}\cite{F}
Let $X$ be a scheme with a cellular decomposition. Then for $0\leq i \leq \dim(X)$, 
\par (i) $H_{2i+1}(X) = 0$ 
\par (2) $H_{2i}(X)$ is a $\mathbb{Z}$-module freely generated by the classes of the closures of the $i$-dimensional cells.
\par (iii) The cycle map $\mathrm{cl}: A_*(X) \to H_*(X)$ is an isomorphism.
\end{thm}

\medskip\noindent
Thus, to count the sum of the Betti numbers of a smooth Hilbert scheme it suffices to count the number of saturated monomial ideals with that Hilbert polynomial.

\section{The projective line}
\label{sec: line}
\noindent
We first consider the case where $n=1$; equivalently, this is the case where the polynomial ring is $R = \mathbb{C}[x_0,x_1]$. In this case, the only possible partitions are $\lambda = (1^m)$ or $(2)$ which are equivalent to the constant Hilbert polynomial $m$ or $t+1$. 
It is well known that $\mathbb{P}^{1[m]} =\mathbb{P}^m$ and $\mathbb{P}^{1[t+1]}$ is a reduced point so Theorem \ref{thm: main} is immediate in these cases.
However, for completeness and clarity, we will give a more basic argument in the case of $\mathbb{P}^{1[m]} =\mathbb{P}^m$ which will illuminate the argument which is somewhat obscured by indexing in the later sections.

\subsection{Translation for the two variable case}
Consider the case of two variables.
Monomials in the variables $x_0$ and $x_1$ are of the form $x_0^a x_1^b$ where $a$, $b \in \mathbb{Z}_{\geq 0}$.
Monomials in two variables are equivalent to points in the lattice $\mathbb{Z}_{\geq 0}^2$ by pairing the point $(a,b)$ with the monomial $x_0^ax_1^b$.
By a \textit{ray} in this lattice we will mean the points corresponding to the monomials in a set of the form 
\[P_{a}^{i_0}=\{x_{i_0}^{a}x_{i_1}^{b}|b\in\Z_{\geq 0}\}\text{ for some fixed }a\in\Z_{\geq0}.\]

\medskip \noindent
We want to see how a set of rays corresponds to (a complement of) a monomial ideal.
Since all saturated ideals in two variables are principle, consider the monomial ideal $I=(x_{0}^{a}x_{1}^{b})$.
Every monomial $\overline{x} = x_0^cx_1^d \not \in I$ either has that  $c<a$ or that $d<b$.
That can be rephrased as \[\overline{x} \not \in I \text{ if and only if } \overline{x} \in \left(\bigcup_{i=0}^{b-1} P_{i}^{1}\right) \bigcup \left(\bigcup_{j=0}^{a-1} P_{j}^{0}\right).\]

\begin{figure}[h]
\centering
\includegraphics[width=100mm]{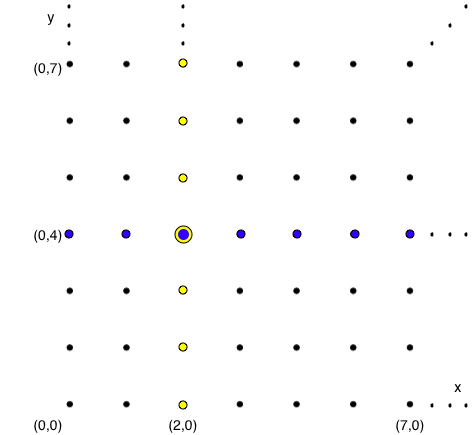}
\caption{A visualization of the sets $P_{4}^{y}$ (in blue) and $P_{2}^{x}$ (in maize) in the $\mathbb{Z}_{\geq0}^{2}$ (x,y) lattice}
  \label{fig:2dlattice}
\end{figure}

Since the Hilbert polynomial of $I = (x_0^ax_1^b)$ is $a+b$, this shows the following lemma.

\begin{lemma}
The number of saturated monomial ideals in two variables with Hilbert polynomial $d$, which corresponds to the partition $(1^d)$, is the same as the number of ways to lay $d$ rays in the lattice $\mathbb{Z}_{\geq 0}^2$ in stacks along the two axes.
\end{lemma}

The next proposition uses this lemma to count the number of saturated monomial ideals.

\begin{prop}\label{prop: line}
Given the ring $R = \mathbb{C}[x_0,x_1]$ and the partition $\lambda = (1^m)$, there are exactly $m+1$ saturated monomial ideals of $R$ with Hilbert polynomial $p_\lambda$.
\end{prop}

\begin{proof}
By the previous lemma, the problem of finding the number of saturated monomial ideals boils down to finding the number of ways of placing $m$ rays along either the $x$-axis or $y$-axis.  We can count this by instead considering our $m$ rays to be $m$ indistinct objects and our axes to be two distinct bins in which we are placing the objects.  The number of ways to place these objects in these bins is a weak composition $m$ in 2 bins which is equivalent to $\binom{m+2-1}{2-1}=m+1$.  
\end{proof}

\section{The projective plane}
\label{sec: plane}
\noindent
In this section, we prove case (1) of Theorem \ref{thm: main}, which is the case of Hilbert schemes on the projective plane.
We note that this is known, at least implicitly by \cite{ES}, in the case of Hilbert schemes of points.
To prove this case, we first show the correspondence of saturated monomial ideals in 3 variables and the placement of rays and quadrants in $\Z_{\geq0}^{3}$ similarly to the correspondence in Section \ref{sec: line}.

\subsection*{The three variable case}
In the three dimensional lattice, we have \textit{rays} of the form
\[P_{a,b}^{i_0,i_1}=\{x_{i_0}^{a}x_{i_1}^{b}x_{i_2}^{c}|c\in\Z_{\geq 0}\}\text{ for some fixed }a,b\in\Z_{\geq0}.\]
On the other hand, we have \textit{quadrants} of the form
\[P_{a}^{i_0}=\{x_{i_0}^{a}x_{i_1}^{b}x_{i_2}^{c}|b,c\in\Z_{\geq 0}\}\text{ for some fixed }a\in\Z_{\geq0}.\]
We want to see how the complement of a monomial ideal in this case is the union of quadrants and rays; that is the content of the following lemma.

\begin{lemma}
\label{lem: plane translation}
Given the ring $R = \mathbb{C}[x_0,x_1,x_2]$ and a saturated monomial ideal \\$I=(x_{0}^{\alpha_{0}^{1}}x_{1}^{\alpha_{1}^{1}}x_{2}^{\alpha_{2}^{1}},...,x_{0}^{\alpha_{0}^{m}}x_{1}^{\alpha_{1}^{m}}x_{2}^{\alpha_{2}^{m}})$, then the set of monomials not in $I$ is of the form \\$\big(\bcup_{i\in\Z} P_{p_{1}^{i},p_{2}^{i}}^{c_{1}^{i},c_{2}^{i}}\big)\bigcup\big(\bcup_{j\in\Z} P_{p_{1}^{j}}^{c_{1}^{j}}\big)$ where 
\[P_{a,b}^{i_0,i_1}=\{x_{i_0}^{a}x_{i_1}^{b}x_{i_2}^{c}|c\in\Z_{\geq 0}\}\text{ for some fixed }a,b\in\Z_{\geq0}.\]
\[P_{a}^{i_0}=\{x_{i_0}^{a}x_{i_1}^{b}x_{i_2}^{c}|b,c\in\Z_{\geq 0}\}\text{ for some fixed }a\in\Z_{\geq0}.\]
These are equivalent to quadrants and rays respectively in the $\Z_{\geq0}^{3}$ lattice where each point in our lattice $(y_{0},y_{1},y_{2})$ corresponds to the monomial $x_{0}^{y_{0}}x_{1}^{y_1}x_{2}^{y_{2}}$. 
\end{lemma}

\begin{proof}First consider the monomial ideal generated by a single monomial $I=(x_{0}^{\alpha}x_{1}^{\beta}x_{2}^{\gamma})$.
A monomial $\overline{x}= x_0^a x_1^b x_2^c$ is not in $I$ if and only if $a<\alpha$, $b<\beta$, or $c<\gamma$.
In other words,
\[\overline{x} \not \in I \text{ if and only if } \overline{x}\in \left(\bcup_{i=0}^{\alpha-1}P_{i}^{0}\right)\cup\left(\bcup_{j=0}^{\beta-1}P_{j}^{1}\right)\cup\left(\bcup_{k=0}^{\gamma-1}P_{k}^{2}\right).\]
Now let's consider a general monomial ideal, which we denote $I=(x_{0}^{\alpha_{1}}x_{1}^{\beta_{1}}x_{2}^{\gamma_{1}},...,x_{0}^{\alpha_{m}}x_{1}^{\beta_{m}}x_{2}^{\gamma_{m}})$ where $\alpha_{l},\beta_{l},\gamma_{l}\in\Z_{\geq0}\;,\;\forall l\in\{1,...,m\}$.  
Defining $I_{l}=(x_{0}^{\alpha_{l}}x_{1}^{\beta_{l}}x_{2}^{\gamma_{l}})$ for $l\in\{1,..,m\}$, we notice that $I=\bcup_{l=1}^{m}I_{l}$.
By De Morgan's Law, we see that 
\[m\not \in I \text{ if and only if }=m \in \bcap_{l=1}^{m}\left[\left(\bcup_{i=0}^{\alpha_{l}-1}P_{i}^{0}\right)\cup\left(\bcup_{j=0}^{\beta_{l}-1}P_{j}^{1}\right)\cup\left(\bcup_{k=0}^{\gamma_{l}-1}P_{k}^{2}\right)\right].\]
To make set manipulation simpler, we define \[A_l=\bcup_{i=0}^{\alpha_{l}-1}P_{i}^{0},B_l=\bcup_{j=0}^{\beta_{l}-1}P_{j}^{1},\text{ and } C_l=\bcup_{k=0}^{\gamma_{l}-1}P_{k}^{2}.\]
Using this, we can rewrite the compliment of $I$ as 
\[\overline{x} \not \in I \text{ if and only if } \overline{x}\in \bcap_{l=1}^{m}\left(A_{l}\cup B_{l}\cup C_{l}\right).\]
Recall that by the distributive property of set intersections over unions this intersection is equivalent to a union of $3^m$ sets of size $m$. 

\medskip\noindent
In order to formally write this, we need some notation.
Recall the symmetric group on $m$ letters, denoted $S_m$, is the set of permutations of the set $\{1,\cdots,m\}$.
Define the subset $G_{2,m}$ of $S_m$ as the set of permutations $\sigma$ of $\{1,\cdots,m\}$ such that $\sigma_1 <\cdots <\sigma_k$, $\sigma_{k+1}<\cdots <\sigma_l$, and $\sigma_{l+1}<\cdots <\sigma_m$ for some $1\leq k \leq l \leq m$, i.e. permutations which decrease at most twice.
With this notation, we can write the complement of $I$ as 
\[\overline{x} \not \in I \text{ if and only if } \overline{x}\in  \bigcup_{g \in G_{2,m}} \left( \left(\bigcap_{j=1}^k A_{g(j)}\right) \bigcap \left(\bigcap_{j=k+1}^l B_{g(j)}\right) \bigcap \left(\bigcap_{j=l+1}^m C_{g(j)}\right)\right).\]

\medskip\noindent
We must consider 3 possible cases for our $m$ sized intersections depending on the values of $k$ and $l$.

\medskip\noindent
\textbf{CASE I:} Consider the case where $0<k<l<m$. An intersection of such a form  will contain finitely many monomials.  Since we are interested only in saturated ideals we can ignore this case since the Hilbert polynomial $H_{I}(d)$ wont consider the elements in this intersection.

\medskip \noindent
\textbf{CASE II:}
Next, consider any $m$-intersection of the form $0=k<l<m$, $0<k=l<m$, or $0<k<l=m$.
If $0<k<l=m$, the intersection is of the following form:
\[A_{i_1}\cap A_{i_2}\cap \cdots \cap A_{i_k}\cap B_{i_{k+1}}\cap \cdots \cap B_{i_{m-1}}\cap B_{i_m}.\]

\medskip\noindent
Let $\theta_{g,0}=\min\{\alpha_{i_1},...,\alpha_{i_k}\}$ and $\theta_{g,1}=\min\{\beta_{i_{k+1}},...,\beta_{i_m}\}$. 
The above intersection is now equivalent to 
\[\left(\bcup_{i=0}^{\theta_{g,0}-1}P_{i}^{0}\right)\cap\left(\bcup_{j=0}^{\theta_{g,1}-1}P_{j}^{1}\right)=\bcup_{0\leq i\leq \theta_{g,0}-1\;,\;0\leq j\leq\theta_{g,1}-1}P_{i,j}^{0,1}.\]
Similarly, if  $0<k=l<m$ and $\theta_{g,0}=\min\{\alpha_{i_1},...,\alpha_{i_k}\}$ and $\theta_{g,2}=\min\{\gamma_{i_{k+1}},...,\gamma_{i_m}\}$, then the intersection above is now equivalent to 
\[\left(\bcup_{i=0}^{\theta_{g,0}-1}P_{i}^{0}\right)\cap\left(\bcup_{j=0}^{\theta_{g,2}-1}P_{j}^{2}\right)=\bcup_{0\leq i\leq \theta_{g,0}-1\;,\;0\leq j\leq\theta_{g,2}-1}P_{i,j}^{0,2}.\]
Finally, if  $0=k<l<m$ and $\theta_{g,1}=\min\{\beta_{i_1},...,\beta_{i_k}\}$ and $\theta_{g,2}=\min\{\gamma_{i_{k+1}},...,\gamma_{i_m}\}$, then the intersection above is now equivalent to 
\[\left(\bcup_{i=0}^{\theta_{g,1}-1}P_{i}^{1}\right)\cap\left(\bcup_{j=0}^{\theta_{g,2}-1}P_{j}^{2}\right)=\bcup_{0\leq i\leq \theta_{g,1}-1\;,\;0\leq j\leq\theta_{g,2}-1}P_{i,j}^{1,2}.\]

\medskip\noindent
\textbf{CASE III:}
Lastly, consider any $m$-intersection with $0<k=l=m$, $0=k<l=m$, or $0=k=l<m$.
If $0<k=l=m$, we have the the intersection
\[A_{1}\cap \cdots \cap A_{m}\]
Now let $\theta_0=\min\{\alpha_{1},\cdots,\alpha_{n}\}$ and note that
\[A_{1}\cap \cdots \cap A_{m}=\bcup_{i=0}^{\theta_0-1}P_{i}^{0}.\]
Similarly, if $0=k<l=m$ and we set $\theta_1=\min\{\beta_{1},\cdots,\beta_{n}\}$, then
\[B_{1}\cap \cdots \cap B_{m}=\bcup_{i=0}^{\theta_1-1}P_{i}^{1}.\]
Finally, if $0=k=l<m$ and we set $\theta_2=\min\{\gamma_{1},\cdots,\gamma_{n}\}$, then
\[C_{1}\cap \cdots \cap C_{m}=\bcup_{i=0}^{\theta_2-1}P_{i}^{2}.\]

\medskip\noindent
Since every intersection is one of those three cases, the result follows.
\end{proof}

\medskip\noindent
This lemma means that counting saturated monomial ideals on the plane with fixed Hilbert polynomial is equivalent to counting quadrants and rays whose complement is the monomials of an ideal with that Hilbert polynomial.
We want to bound the exact number of quadrants and rays that will give the correct Hilbert polynomial, but first we need to know some more information about Hilbert polynomials when $n=2$. 

\begin{proposition}
\label{prop: planeCoefficients}
Suppose $p(d)$ is a polynomial in $d$ such that $p(d)=Md-r$ for some $M,r \in \R$, then $p(d)$ is a Hilbert polynomial if and only if $M\in \Z_{\geq 0}$, $r\in \Z$ and
\[r\leq \frac{M^2-3M}{2}\]
and therefore $\lambda$ takes the form 
\[\lambda = (2^{M},1^{\frac{M^2-3M}{2}-r}).\]
\end{proposition}
\begin{proof}
We first prove the forward direction.
Assume $p(d)$ is a Hilbert polynomial. 
By Macaulay \cite{M}, this implies that there exists a lambda sequence, say $\lambda=(2^{A},1^{B})$, where $A,B\in \Z_{\geq0}$ such that
\[p(d)=\sum_{i=1}^{A}\binom{d+2 -i}{1} +\sum_{i=A+1}^{A+B}\binom{d+1-i}{0}\]
Simplifying the expression for $p(d)$, we get
\[p(d)=Ad-\left(\frac{A^2-3A}{2}-B\right)=Md-r\]
Thus, we see here that $M=A$ and $r=\frac{A^2-3A}{2}-B$. Since $A,B\in\Z_{\geq0}$ and $p(d)$ is a Hilbert polynomial, we have that $M \in \Z_{\geq0}$ and $r=B-\frac{A(A-3)}{2} \in \Z$.
Since $B\geq0$, we get that $\frac{M^2-3M}{2}\geq r$.

\medskip\noindent
Now we proceed with the reverse direction.
Let $p(d)=Md-r$, $M\in \Z_{\geq0}$, $r\in \Z$, and
$r\leq \frac{M^2-3M}{2}$.
Then $p(d)$ is  a Hilbert polynomial because we can choose 
$\lambda=(2^{M},1^{\frac{M^2-3M}{2}-r})$.
With that choice, we get that
\[H_{\lambda}(d)= \sum_{i=1}^{\frac{M^2-3M}{2}-r+M}\binom{d+\lambda_i-i}{\lambda_i-1} = \sum_{i=1}^{M}(d+2-i) + \sum_{i=M+1}^{\frac{M^2-3M}{2}-r+M}(1) = Md-r=p(d).\]
\end{proof}

\medskip\noindent
Next, we know from Lemma \ref{lem: plane translation} that the number of saturated monomial ideals in three variables for a given Hilbert polynomial $p(d)$ is equivalent to the number of ways of laying a some number of quadrants and rays in the $\Z_{\geq0}^3$ lattice. So, now we establish the connection between a given Hilbert polynomial and the exact number of quadrants and rays that will be placed in the $\Z_{\geq0}^{3}$ lattice.

\begin{lemma}
\label{lem: plane equiv}
In $R=\mathbb{C}[x_0,x_1,x_2]$, given a Hilbert polynomial $p_\lambda$ with associated lambda partition $\lambda=(2^{A},1^{B})$, the number of saturated monomial ideals in $R$ with associated Hilbert polynomial $p_\lambda$ is equivalent to the number of ways of stacking $A$ quadrants and $B$ rays in $\Z_{\geq0}^{3}$.
\end{lemma}

\begin{proof}
Consider a Hilbert polynomial $p_\lambda$. By Theorem 4.2 above we note that this Hilbert polynomial must take the form $p_\lambda(d)=Md-r$ where $A= M\in\Z_{\geq0}$, $r\in\Z$, and $B = \frac{M^2-3M}{2} - r \geq 0$. 

\medskip\noindent
Any quadrant in $\Z_{\geq0}^3$, say $P_{a}^{i_0}$ represents all monomials of the form $x_{i_0}^{a}x_{i_1}^{\rho}x_{i_2}^{\eta}$ where $a$ is fixed and $\rho,\eta \in \Z_{\geq0}$ are variable. We will consider listing the quadrants in order subject to listing $P_{a}^{i_j}$ before $P_{a+1}^{i_j}$ for all $a$ and $j$.
Thus, the number of monomials of degree d in which $x_{i_0}$ has degree a is $\binom{d-a+1}{1}=d-a+1$. 
However, if $P_{a}^{i_j}$ is the $i$-th quadrant we list, then it has $i-a$ monomials of degree $d$ in common with previously listed quadrants (one for each quadrant which is not parallel to it). Putting this together, the $i$-th listed quadrant contributes $d-a-(i-1-a)+1 = d+2-i$ to the Hilbert polynomial.
Note, by the \textit{contribution} of a ($k$-)orthant we mean the number of monomials of degree $d$ for $d>>0$ in that ($k$-)orthant.
This implies that in order for the linear term of the Hilbert polynomial $p_\lambda(x)$ to have a coefficient of $M$, there must be $M$ distinct quadrants in the complement of the ideal as the rays do not contribute to the linear term.
Then the contribution of all of the quadrants to the Hilbert polynomial is $\sum_{i=1}^M \binom{t+2-i}{2-1}$. 

\medskip\noindent
Then the rays must contribute the remaining part of the constant term which is easily seen to be $B$.
Since rays in $\Z_{\geq0}^3$ contain only one point in the $\Z_{\geq0}^{3}$ lattice for each $d>0$, they contribute 1 each to the Hilbert polynomial.
Thus, the monomials not in the ideal consist of exactly $A$ quadrants and $B$ rays.
\end{proof}

\medskip\noindent
Thus, given a Hilbert polynomial $p_{\lambda}$ with associated lambda partition $\lambda=(2^{A},1^{B})$, $A,B\in \Z_{\geq0}$ we note that the number of saturated monomial ideals for this Hilbert polynomial in $R=\mathbb{C}[x_0,x_1,x_2]$ is equivalent to number of ways of stacking $A$ quadrants and $B$ rays in $\Z^3_{\geq0}$.
Finally, we can establish the following proposition, which is case (3) of Theorem \ref{thm: main}.

\begin{proposition}\label{prop: 1}
Given the ring $R = \mathbb{C}[x_0,x_1,x_2]$ and the partition $\lambda = (2^{m},1^{r})$, the number of saturated monomial ideals of $R$ with Hilbert polynomial $p_\lambda$ is exactly
\[\binom{m+2}{2}\cdot \sum_{c_1+c_2+c_3=r}\Big[f_1(c_1)\cdot f_1(c_2)\cdot f_1(c_3)\Big]\]
where $c_1,c_2,c_3 \in \Z_{\geq0}$ and 
$f_1: \Z_{\geq 0}\to \Z_{>0}$ 
is the function which maps an integer $c$ to the number of integer partitions of $c$.
\end{proposition}

\begin{proof}
First, consider the case where $r=0$. 
In that case, the lambda partition of $p_\lambda$ is of the form $\lambda=(2^{m})$. 
By the previous lemma, we know that the number of saturated ideals $I\in R=\mathbb{C}[x_0,x_1,x_2]$ is equivalent to the number of ways to lay $m$ quadrants in $\Z^{3}_{\geq0}$, or in other words, the number of unique sets
\[\left(\bigcup_{i=0}^{\alpha-1}P_i^{0}\right)\cup\left(\bigcup_{j=0}^{\beta-1}P_j^{0}\right)\cup\left(\bigcup_{k=0}^{\psi-1}P_k^{0}\right)\]
where $\alpha+\beta+\psi=m$.\\

\medskip\noindent
Clearly, the uniqueness of each set is determined by the assignment of $\alpha,\beta$ and $\psi$. 
Thus, for the case where $r=0$, we get that the number of saturated monomial ideals for $p_{\lambda}$ for $n=2$ is
$\binom{m+2}{2}$.

\medskip\noindent
Next, consider when $m=0$, i.e. when $\lambda=(1^{r})$. 
By the previous lemma, the number of saturated monomial ideals with Hilbert polynomial $p_{\lambda}$ for $n=2$ is equivalent to the number of ways to stack $r$ rays in $\Z^3_{\geq0}$. 
Much like for the case where $r=0$, we note that there are only 3 forms of which these rays can be; for a fixed $a,b \in\Z_{\geq0}$ these are
\[P_{a,b}^{i_0,i_1}=\{x_{i_0}^ax_{i_1}^b{x_{i_2}^c | c\in \Z_{\geq0} }\}, P_{a,b}^{i_0,i_2}=\{x_{i_0}^ax_{i_1}^c{x_{i_2}^b | c\in \Z_{\geq0} }\}\text{, or } P_{a,b}^{i_1,i_2}=\{x_{i_0}^cx_{i_1}^a{x_{i_2}^b | c\in \Z_{\geq0} }\}.\]

\medskip\noindent
Visually, these $r$ rays can only extend in one of 3 directions in the $\Z_{\geq0}^3$ lattice: the $x_{i_0}$, the $x_{i_1}$, or the $x_{i_2}$ direction. 
In order to count the number of ways to pick $r$ rays, we first count the ways to divide the rays into one of the 3 directions, and then count the number of ways to orient the rays in each direction.  
These two counts are independent of each other.
It is easy to see that there is $\binom{r+2}{2}$ unique distributions of these $r$ rays into the $x_{i_0},x_{i_1}$ and $x_{i_2}$ directions. 

\medskip\noindent
If $P_{a,b}^{i_j,i_k}$ is in the complement of a saturated ideal then so are $P_{a-1,b}^{i_j,i_k}$ or $P_{a-1}^{i_j}$ unless $a=0$. 
Similarly, if $P_{a,b}^{i_j,i_k}$ is in the complement of a saturated ideal then so are $P_{a,b-1}^{i_j,i_k}$ or $P_{b-1}^{i_k}$ unless $b=0$. 
Given the previous facts and that we have already picked which quadrants are included, any valid choice of a set of rays in one direction is such that the lattice points of fixed degree in the rays form (the centers of the squares of) a Young diagram. 
In other words the number of valid arrangement of $k$ rays in one direction is the number of integer partitions of $k$.\\ 
If we do this for each distinct distribution of rays then number of ways to stack $r$ rays in $\Z_{\geq0}^3$ is given by
\[\sum_{c_0+c_1+c_2=r}\big[f_1(c_0)\cdot f_1(c_1)\cdot f_1(c_2) \big]\]

\medskip\noindent
Lastly, when $m$ and $r$ are both non-zero, we find that the choice of the $m$ quadrants will not affect the number of choices of the $r$ rays. To see this, recall that a given choice of the $m$ quadrants in $\Z_{\geq0}^3$ is of the form
\[\Bigg(\bigcup_{i=0}^{\alpha-1}P_i^{0}\Bigg)\cup\Bigg(\bigcup_{j=0}^{\beta-1}P_j^{0}\Bigg)\cup\Bigg(\bigcup_{k=0}^{\psi-1}P_k^{0}\Bigg)\]
where $\alpha+\beta+\psi=m$.

\medskip \noindent
This choice `shifts' the region in which the remaining $r$ rays can be placed. We can think about choosing the remaining $r$ rays in the $\Z_{\geq0}^3$ lattice in which the coordinate axes are defined by $x_0+\alpha, x_1+\beta$ and $x_2+\psi$. Thus, the number of ways to choose these remaining $r$ rays in this region is exactly the same as choosing $r$ rays when $m=0$, Thus, the placing of the $m$ quadrants is independent to the stacking of the $r$ rays. Thus, combining the earlier two cases, and by the general counting principle we have that there are 
\[\binom{m+2}{2}\cdot \sum_{c_1+c_2+c_3=r}\Big[f_1(c_1)\cdot f_1(c_2)\cdot f_1(c_3)\Big]\]
where $c_1,c_2,c_3 \in \Z_{\geq0}$ ways to stack $m$ quadrants and $r$ rays in $\Z_{\geq0}$. Along with Lemma \ref{lem: plane equiv}, this concludes the proof.
\end{proof}

\section{General Case}
\label{sec: general}
\noindent
In this section, we prove the remaining cases of Theorem \ref{thm: main}.

\medskip\noindent
We first note that case (7) of Theorem \ref{thm: main} is immediate as those Hilbert schemes are just a single reduced point.

\medskip\noindent
For the rest of the cases, we must first show that a saturated monomial ideal corresponds to a union of rays, orthants, 3-orthants, etc. in the the lattice $\mathbb{Z}_{\geq 0}^{n+1}$.
\begin{prop}
Given the ring $R = \mathbb{C}[x_0,x_1,\dots,x_n]$ and a saturated monomial ideal\\ $I=(x_{0}^{\alpha_{0}^{1}}\dots x_{n}^{\alpha_{n}^{1}},\dots,x_{0}^{\alpha_{0}^{m}}\dots x_{n}^{\alpha_{n}^{m}})$, then the set of monomials not in $I$ is of the form $\bcup_{i=1}^M P_{p_{1}^{i},\dots,p_{r_i}^{i}}^{c_{1}^{i},\dots,c_{r_i}^{i}}$ for some positive integer $M$  for some fixed $p_{1},\dots,p_{n+1-k}\in\Z_{\geq0}$ where 
\[P_{p_{1},\dots,p_{n+1-k}}^{c_{1},\dots,c_{n+1-k}}=\{x_{c_{1}}^{p_{1}}\dots x_{c_{n+1-k}}^{p_{n+1-k}}x_{c_{n+2-k}}^{p_{n+2-k}}\dots x_{c_{n}}^{p_{n}}|p_{n+2-k},\dots ,p_{n}\in\Z_{\geq0}\}\]
is a $k$-orthant in a $n+1$ dimensional lattice where each point in our lattice $(y_{0},\dots,y_{n})$ corresponds to the monomial $x_{0}^{y_{0}}\dots x_{n}^{y_{n}}$. 
\end{prop}

\begin{proof}
\noindent
Utilizing the same approach as before give us that a $k$-orthant in the $n+1$ dimensional lattice would be of the form
$$P_{p_{1},\dots,p_{n+1-k}}^{c_{1},\dots,c_{n+1-k}}=\{x_{c_{1}}^{p_{1}}\dots x_{c_{n+1-k}}^{p_{n+1-k}}x_{c_{n+2-k}}^{p_{n+2-k}}\dots x_{c_{n}}^{p_{n}}|p_{n+2-k},..,p_{n}\in\Z_{\geq0}\}\text{ for some fixed }p_{1},\dots,p_{n+1-k}\in\Z_{\geq0}$$
As with the previous approach let's begin by considering a monomial ideal generated by a single element $(x_{0}^{\alpha_{0}}\dots x_{n}^{\alpha_{n}})$. 
Then the monomials in the compliment of $I$ is 
\[\left(\bcup_{i_{0}=1}^{\alpha_{0}-1}P_{i_{0}}^{0}\right)\cup\dots\cup\left(\bcup_{i_{n}=1}^{\alpha_{n}-1}P_{i_{n}}^{n}\right).  \]
Now, let's consider a monomial ideal with $m$ generators $I=(x_{0}^{\alpha_{0}^{1}}\dots x_{n}^{\alpha_{n}^{1}},\dots,x_{0}^{\alpha_{0}^{m}}\dots x_{n}^{\alpha_{n}^{m}})$ and let $I_{\ell}=(x_{0}^{\alpha_{0}^{\ell}}\dots x_{n}^{\alpha_{n}^{\ell}})$ which then gives us that $I=\bcup_{\ell=1}^{m}I_{\ell}$.
By De Morgan's laws, we have that $I$'s complement in the lattice is
\[\overline{x} \not \in I \text{ if and only if } \bar{x}\in \bcap_{\ell=1}^{m}\left[\left(\bcup_{i_{0}=1}^{\alpha_{0}^{\ell}-1}P_{i_{0}}^{0}\right)\cup\dots\cup\left(\bcup_{i_{n}=1}^{\alpha_{n}^{\ell}-1}P_{i_{n}}^{n}\right)\right]\]
From here, let's simply our notation by setting 
\[A_{\ell}^{j}= \bcup_{i_{j}=1}^{\alpha_{j}^{\ell}-1}P_{i_{j}}^{j}.\]
Using this, we can rewrite the compliment of $I$ as 
\[\overline{x} \not \in I \text{ if and only if } \overline{x}\in \bcap_{\ell=1}^{m}\left(\bigcup_{j=1}^n A_{\ell}^j\right).\]
The intersection above then simplifies into a $(n+1)^{m}$ sized union of $m$ sized intersections.  

\medskip\noindent
In order to formally write this, we need some notation.
Recall the symmetric group on $m$ letters, denoted $S_m$, is the set of permutations of the set $\{1,\dots,m\}$.
Define the subset $G_{n,m}$ of $S_m$ as the set of permutations $\sigma$ of $\{1,\dots,m\}$ such that $\sigma_1 <\dots <\sigma_{k_1}$, $\dots$, and $\sigma_{k_{n-1}+1}<\dots <\sigma_n$ for some $1\leq k_1 \leq \dots \leq k_{n-1} \leq n$, i.e. permutations which decrease at most $n$ times.
With this notation, we can write the complement of $I$ as 
\[\overline{x} \not \in I \text{ if and only if } \overline{x}\in  \bigcup_{g \in G_{n,m}} \left( \left(\bigcap_{j=1}^{k_1} A_{g(j)}^1\right) \bigcap \dots \bigcap \left(\bigcap_{j=k_{n-1}+1}^m A_{g(j)}^{n}\right)\right).\]

\medskip\noindent
Now as before,  we observe that if a single one of these intersections has all possible $A^{i}$, then it contains finitely many monomials.  
Hence, we can ignore this case since for $d>>0$ the Hilbert polynomial $H_{I}(d)$ won't count these.
Next, consider an intersection of the form 
\[(A_{1}^{0}\cap\dots\cap A_{j_{1}}^{0})\cap (A_{j_{1}+1}^{1}\cap\dots\cap A_{j_{2}}^{1})\cap\dots\dots\cap (A_{j_{h}+1}^{h}\cap\dots\cap A_{j_{h+1}}^{h})\]
with $h<n$. 
If we let $\theta_{0}=\min\{\alpha_{0}^{1},\dots,\alpha_{0}^{j_{1}}\},\theta_{1}=\min\{\alpha_{1}^{j_{1}+1},\dots,\alpha_{1}^{j_{2}}\},\dots,\theta_{h}=\min\{\alpha_{h}^{j_{h}+1},\dots,\alpha_{h}^{j_{h+1}}\}$ then our intersection above is equivalent to
\[\left(\bcup_{l_{0}=0}^{\theta_{0}-1}P_{l_{0}}^{j_0}\right)\cap\dots\cap\left(\bcup_{l_{h}=0}^{\theta_{h}-1}P_{l_{h}}^{j_h}\right)=\bcup_{0\leq l_{0}\leq\theta_{0}-1,\dots,0\leq l_{h}\leq\theta_{h}-1}P_{l_{0},\dots,l_{h}}^{j_0,..,j_h}.\]
Thus, the intersections of this form correspond to an $n-h$-orthant.
Since any intersection in the union has this form for some $1\leq h \leq n-1$ (or is irrelevant to the saturation), the result follows.
\end{proof}

\medskip\noindent
Based on the case of the line and the orthant, one may naively expect that any arrangement of ($k$-)orthants giving the Hilbert polynomial $p_\lambda$ where $\lambda = (n^{k_n},\dots,1^{k_1})$ consists of $k_i$ many $i$-orthants, but this is incorrect.

\begin{exmp}\label{ex: 2 quadrants}
Consider the Hilbert polynomial $p(d)=2d+1$ with lambda sequence $\lambda(2^{[2]})$ and $n\geq 3$. 
If the naive correspondence holds true for all $n\geq3$, then we should expect that every choice of two quadrants in $\Z_{\geq0}^{n+1}$ will result in a Hilbert function $h(d)$ such that $h(d)=p(d)$ for some $d>>0$. 
However, consider the following two quadrants in $\Z_{\geq0}^{n+1}$ 
\[P_1=P_{0,0,0,...,0}^{2,3,4,...,n} \;\;\text{and}\;\; P_2=P_{0,0,0,...,0}^{0,1,4,...,n}\;\; \]

\medskip\noindent
Note that $P_1$ and $P_2$ do not intersect other than at the origin since recall that the elements of $P_1$ and $P_2$ are
\[P_1=\{1,x_0,x_1,x_0x_1,...,x_0^ax_1^b,...\} \;\;\text{for some } \;\; a,b\in \Z_{\geq0}\]
\[P_2=\{1,x_2,x_3,x_2x_3,...,x_2^cx_3^d,...\} \;\;\text{for some } \;\; c,d\in \Z_{\geq0}\]
Thus, $P_1$ and $P_2$ each contribute $d+1$ to the Hilbert polynomial for this specific monomial ideal $h(d)$.  
Thus, the Hilbert polynomial of the corresponding ideal must be $h(d)=(d+1)+(d+1)=2d+2$. 
Thus, the naive correspondence fails.
\end{exmp}

\medskip\noindent
In order to study the case where the naive correspondence does hold, or where we can salvage an adapted version of it, we first need a general remark

\begin{remark}
\label{rem: coeff}
Consider a $k$-orthant $K$ in the complement of a saturated monomial ideal $I$. 
The number of degree $d$ monomials in $K$ is $\binom{d+k-1}{k-1}$ so $K$'s contribution to the Hilbert polynomial only effects the coefficients of the terms with degree at most $k-1$ and adds nonnegatively to the degree $k-1$ coefficient.
\end{remark}

\medskip\noindent
Using this remark, we can show that the naive correspondence does hold for the $n$-orthants in the arrangement and the $n^k$ in the partition.

\begin{lemma}
\label{lem: northants}
Let $p_\lambda$ be the Hilbert polynomial corresponding to the partition $\lambda = (n^k,\dots)$.
Then the complement of any saturated monomial ideal in $n+1$ variables with Hilbert polynomial $p_\lambda$ contains exactly $k$ many $n$-orthants.
\end{lemma}

\begin{proof}
Given $\lambda$, the corresponding Hilbert polynomial $\sum_{i=1}^m \binom{t+\lambda_i-i}{\lambda_i-1}$ has degree $n-1$.
Since there are no terms of degree greater than $n-1$, the complement of the ideal cannot contain the entire $n+1$ dimensional lattice so only the $n$-orthants contained in the complement of $I$ contribute to the leading coefficient.
Since the $n$-orthants overlap with other ($k$-)orthants in lower dimensional spaces, each $n$-orthant contributes $\frac{1}{(n-1)!}$ to the leading coefficient.
Since the leading coefficient of $p_\lambda$ is $\frac{k}{(n-1)!}$, the complement of $I$ contains exactly $k$ many $n$-orthants.
\end{proof}

\medskip\noindent
These lemmas forms the basis for proving the remaining cases of Theorem \ref{thm: main}.

\begin{prop}
\label{prop: 6}
(6) If $\lambda =(n^k,1^3)$, then $H_{n,\lambda} = \binom{k+n}{n}\frac{n(n+1)(5n+1)}{3}$.
\end{prop}

\begin{proof}
In this case, $p_\lambda(d) = \sum_{i=1}^k \binom{d+n-i}{n-1} + \sum_{i=k+1}^{k+3}\binom{d+1-i}{1-1} = \sum_{i=1}^k \binom{d+n-i}{n-1} + 3$.
By Lemma \ref{lem: northants}, there are $k$ many $n$-orthants contained in the complement of $i$.
If we think of choosing them in order by the amount which they are shifted away from the axes, consider the $i$-th chosen $n$-orthant, $K$. It is parallel to $j$ previously chosen $n$-orthants for some $0\leq j \leq i-1$.
Then there are $\binom{d+n-j}{n-1}$ monomials of degree $d$ in $K$. 
Similarly, it overlaps each of the $k$-th previously chosen nonparallel $n$-orthants (of the $i-1-j$) in $\binom{d+n-1-k}{n-2}$ many monomials of degree $d$.
Thus, the contribution of the $i$-th chosen $n$-orthant is $\binom{d+n-j}{n-1} - \sum_{k=1}^{i-j}\binom{d+n-k-j}{n-2} = \binom{d+n-i}{n-1}$.
Thus, the contribution of all of the $n$-orthants to the Hilbert polynomial is $ \sum_{i=1}^k \binom{d+n-i}{n-1}$.

\medskip\noindent
This leaves only $3$ to be contributed by other $k$-orthants to the Hilbert polynomial.
Since rays contribute $1$ each and any $k$-orthant for $k\geq 2$ would change higher order terms, this means that there are $3$ rays to be chosen.
Thus, a saturated monomial with this Hilbert polynomial is equivalent to the choice of $k$ many $n$-orthants and 3 rays.
Choosing an $n$-orthant to add is equivalent to choosing which variable to multiply by so there is $\binom{n+k}{k}$ ways to choose $k$ of them.
Given a choice of $n$-orthants, the rays can be chosen so that none are parallel, so that exactly 2 are parallel, or so that all 3 are parallel.

\medskip\noindent
There are $\binom{n+1}{3}$ ways to chose the rays in distinct directions.
There are $2n\binom{n+1}{2}$ ways to chose the rays in only two directions.
Finally, there are $(n+1)\left(n +\binom{n}{2}\right)$ ways to chose the three rays in the same direction.

\medskip\noindent
Putting these two counts together, there are \[\binom{k+n}{n}\Bigg[\binom{n+1}{3}+2n\binom{n+1}{2}+\big(n+1\big)\Big(n+\frac{n(n-1)}{2}\Big)\Bigg]\]
saturated monomial ideals in $n+1$ variables with Hilbert polynomials $p_\lambda$.
This count simplifies to 
\[H_{n,\lambda} = \binom{k+n}{n}\left(\frac{n(n+1)(5n+1)}{3}\right)\]

\end{proof}

\begin{prop}
\label{prop: 3}
(3) If $\lambda = (1)$ or $\lambda = (n^{r-2},\lambda_{r-1},1)$ where $r \geq 2$ and $m\geq \lambda_{r-1} \geq 1$, then $H_{n,\lambda} = n+1$ and $H_{n,\lambda} = \binom{n+r-2}{r-2}\binom{n+1}{\lambda_{r-1}}(n+1)$, respectively.
\end{prop}

\begin{proof}
\noindent
In the case $\lambda = (1)$, the Hilbert polynomial is $p_\lambda = 1$. Then the Hilbert scheme is $\mathbb{P}^n$ so the result is immediate.

\medskip\noindent
In the second case, $p_\lambda(d) = \sum_{i=1}^{r-2} \binom{d+n-i}{n-1} + \binom{d+\lambda_{r-1}-(r-1)}{\lambda_{r-1}-1} +1$.
Analogous to the proof of the Proposition \ref{prop: 6}, the complete of the ideal must contain exactly $r-2$ many $n$ orthants and their contribution to the Hilbert polynomial is $ \sum_{i=1}^{r-2} \binom{d+n-i}{n-1}$.

\medskip\noindent
Since this leaves $\binom{d+\lambda_{r-1}-(r-1)}{\lambda_{r-1}-1} +1$ left to be contributed to the Hilbert polynomial which has degree $\lambda_{r-1}-1$, the complement of the ideal must contain at least 1 $\lambda_{r-1}$-orthant.
By similar argument to choosing the $k$-th $n$-orthant in the previous proof, the first $\lambda_{r-1}$-orthant chosen contributes $\binom{d+\lambda_{r-1}-(r-1)}{\lambda_{r-1}-1}$ to the Hilbert polynomial.
After having chosen this, there is only a remaining $1$ to be contributed to the Hilbert polynomial, which again must be contributed by a ray.
Thus, a saturated monomial ideal with Hilbert polynomial $p_\lambda$ is equivalent to picking $r-2$ many $n$-orthants, a $\lambda_{r-1}$-orthant, and a ray.

\medskip\noindent
As in the previous proof, there is $\binom{n+r-2}{r-2}$ ways to pick those $n$ orthants.
There is then $\binom{n+1}{\lambda_{r-1}}$ ways to choose that $\lambda_{r-1}$-orthant given the chosen $n$-orthants.
There are then $n+1-\lambda_{r-1}$ ways to choose the ray not parallel to the $\lambda_{r-1}$ orthant and $\lambda_{r-1}$ ways to choose the ray parallel to it given the previous choices.
Putting these together gives \[H_{n,\lambda} = \binom{n+r-2}{r-2}\binom{n+1}{\lambda_{r-1}}(n+1).\]
\end{proof}

\begin{prop}\label{prop: 5}
(5) If $\lambda = (n^{r-s-5},2^{s+4},1)$ where $r-5 \geq s \geq 0$, then \[H_{n,\lambda} = \binom{n+r-s-5}{r-s-5}\left(\binom{n+1}{3}\left(\binom{3+s+4}{s+4}-3\right)+\binom{n+1}{2}\right)(n+1).\]
\end{prop}

\begin{proof}
By Lemma \ref{lem: northants}, the complement of a saturated monomial ideal with Hilbert polynomial $p_\lambda$ contains $r-s-5$ many $n$-orthants.
Arguing analogously to the previous cases using the coefficients of the Hilbert polynomials, we see the complement must also contain exactly $s+4$ quadrants.
However, a generalization of Example \ref{ex: 2 quadrants} to 8 variables (choosing 4 quadrants such that they all only intersect at the origin) shows that not all choices of $s+4$ quadrants live in the complement some ideal with the correct Hilbert polynomial.
We now enumerate those that do.

\medskip\noindent
If the $s+4$ quadrants are not all contained in any 3-orthant, then we claim that they do not give the correct Hilbert polynomial.
In particular, if no 3 are in the same 3-orthant, then the first four quadrants we choose contribute at least 
\[\binom{d+2-(r-s-5+1)}{2-1}+\binom{d+2-(r-s-5+2)}{2-1}+\binom{d+2-(r-s-5+3)}{2-1}+1+\binom{d+2-(r-s-5+4)}{2-1}+1.\]
If at least 3 are contained in one 3-orthant, then we choose these three first and a quadrant not contained in that 3-orthant 4th.
These four quadrants contribute \[\binom{d+2-(r-s-5+1)}{2-1}+\binom{d+2-(r-s-5+2)}{2-1}+\binom{d+2-(r-s-5+3)}{2-1}+\binom{d+2-(r-s-5+4)}{2-1}+2.\]
In both cases, all of the subsequently chosen quadrants contribute at least the expected amount.
In either case, this would force a negative number of rays in order to have the correct Hilbert polynomial, which is obviously impossible.

\medskip\noindent
On the other hand, if all of the $s+4$ quadrants are contained in a 3-orthant, then the $k$-th chosen one contributes $\binom{d+2-(r-s-5+k)}{2-1}$ to the Hilbert polynomial as expected.

\medskip\noindent
Finally, given a choice of $r-s-5$ many $n$-orthants and $s+4$ quadrants in some $3$-orthant, there is a remaining one in the Hilbert polynomial which must be contributed by a single additional ray.
Thus, a saturated monomial ideal with Hilbert polynomial $p_\lambda$ is equivalent to the choice of $r-s-5$ many $n$-orthants, $s+4$ quadrants in some $3$-orthant, and a single ray.

\medskip\noindent
As in previous proofs,there are $\binom{n+r-s-5}{r-s-5}$ many ways to chose the $n$-orthants.
There at $\binom{n+1}{3}$ ways to chose the 3-orthant containing the quadrants.
Given a chosen 3-orthant, there is $\binom{3+s+4}{s+4}-3$ ways to chose the quadrants such that they are not all parallel.
We count the case where all of the quadrants are parallel separately as each such choice lies in multiple 3-orthants; there is $\binom{n+1}{2}$ ways to choose them all parallel.
Finally, in any of the above choices, there are $n+1$ ways to choose the remaining ray.
Putting this together gives \[H_{n,\lambda} = \binom{n+r-s-5}{r-s-5}\left(\binom{n+1}{3}\left(\binom{3+s+4}{s+4}-3\right)+\binom{n+1}{2}\right)(n+1) \]
\end{proof}

\begin{prop}
\label{prop: 4}
(4) If $\lambda = (n^{r-s-3},\lambda_{r-s-2}^{s+2},1)$ where $r-3 \geq s \geq 0$ and $n-1 \geq \lambda_{r-s-2} \geq 3$, then \[H_{n,\lambda} = \binom{n+r-s-3}{r-s-3}\left(\binom{n+1}{\lambda_{r-s-2}+1}\left(\binom{\lambda_{r-s-2}+1+s+2}{s+2}-(\lambda_{r-s-2}+1)\right)+\binom{n+1}{\lambda_{r-s-2}}\right)(n+1) .\]
\end{prop}

\begin{proof}
By Lemma \ref{lem: northants}, the complement of a saturated monomial ideal with Hilbert polynomial $p_\lambda$ contains $r-s-3$ many $n$-orthants.
Arguing analogously to the previous cases using the coefficients of the Hilbert polynomials, we see the complement must also contain exactly $s+2$ many $\lambda_{r-s-2}$-orthants.
As before, not all choices of $s+2$  many $\lambda_{r-s-2}$-orthants live in the complement some ideal with the correct Hilbert polynomial.
We now enumerate those that do.

\medskip\noindent
If the $\lambda_{r-s-2}$-orthants are not all contained in any $(\lambda_{r-s-2}+1)$-orthant, then we claim that they do not give the correct Hilbert polynomial.
In particular, choose the first two of them to not lie in the same $(\lambda_{r-s-2}+1)$-orthant, then they contribute at least 
\[\binom{d+\lambda_{r-s-2}-(r-s-3+1)}{\lambda_{r-s-2}-1}+
\binom{d+\lambda_{r-s-2}-(r-s-3+2)}{\lambda_{r-s-2}-1}\]\[+
\left(\binom{d+\lambda_{r-s-2}-(r-s-3+2)}{\lambda_{r-s-2}-2}-
\binom{d+\lambda_{r-s-2}-(r-s-3+2)}{\lambda_{r-s-2}-3}\right)\]
and all of the subsequently chosen ones contribute at least the expected amount.
In either case, this would force a negative number of $(\lambda_{r-s-2}-1)$-orthants in order to have the correct Hilbert polynomial, which is obviously impossible.

\medskip\noindent
On the other hand, if all of them are chosen within a $(\lambda_{r-s-2}+1)$-orthant, then the $k$-th chosen one contributes $\binom{d+\lambda_{r-s-2}-(r-s-2+k)}{\lambda_{r-s-2}-1}$ to the Hilbert polynomial as expected.

\medskip\noindent
Finally, given a choice of $r-s-3$ many $n$-orthants and $s+2$ many $\lambda_{r-s-2}$-orthants in some $(\lambda_{r-s-2}+1)$-orthant, there is a remaining one in the Hilbert polynomial which must be contributed by a single additional ray.
Thus, a saturated monomial ideal with Hilbert polynomial $p_\lambda$ is equivalent to the choice of $r-s-3$ many $n$-orthants, $s+2$ many $\lambda_{r-s-2}$-orthants in some $(\lambda_{r-s-2}+1)$-orthant, and a single ray.

\medskip\noindent
As in previous proofs,there are $\binom{n+r-s-3}{r-s-3}$ many ways to chose the $n$-orthants.
There at $\binom{n+1}{\lambda_{r-s-2}+1}$ ways to chose the $(\lambda_{r-s-2}+1)$-orthant containing the $\lambda_{r-s-2}$-orthants.
Given a chosen $(\lambda_{r-s-2}+1)$-orthant, there are $\binom{\lambda_{r-s-2}+1+s+2}{s+2}-(\lambda_{r-s-2}+1)$ ways to chose the $\lambda_{r-s-2}$-orthants such that they are not all parallel.
We count the case where all of the $\lambda_{r-s-2}$-orthants are parallel separately as each such choice lies in multiple $(\lambda_{r-s-2}+1)$-orthants; there is $\binom{n+1}{\lambda_{r-s-2}}$ ways to choose them all parallel.
Finally, in any of the above choices, there are $n+1$ ways to choose the remaining ray.
Putting this together gives \[H_{n,\lambda} = \binom{n+r-s-3}{r-s-3}\left(\binom{n+1}{\lambda_{r-s-2}+1}\left(\binom{\lambda_{r-s-2}+1+s+2}{s+2}-(\lambda_{r-s-2}+1)\right)+\binom{n+1}{\lambda_{r-s-2}}\right)(n+1).\]
\end{proof}
%
%
%

\bibliographystyle{alphanum}
\bibliography{fullbibliography}

\begin{thebibliography}{LQW}

\bibitem[ByB]{BB}
A.~Bia\l~ynicki Birula.
\newblock Some theorems on actions of algebraic groups.
\newblock {\em Ann. of Math. (2)}, 98:480--497, 1973.

\bibitem[ESm]{ES}
Geir Ellingsrud and Stein~Arild Str\o~mme.
\newblock On the homology of the {H}ilbert scheme of points in the plane.
\newblock {\em Invent. Math.}, 87(2):343--352, 1987.

\bibitem[Eva]{E}
L.~Evain.
\newblock The {C}how ring of punctual {H}ilbert schemes on toric surfaces.
\newblock {\em Transform. Groups}, 12(2):227--249, 2007.

\bibitem[Ful]{F}
William Fulton.
\newblock {\em Intersection theory}, volume~2 of {\em Ergebnisse der Mathematik
  und ihrer Grenzgebiete. 3. Folge. A Series of Modern Surveys in Mathematics
  [Results in Mathematics and Related Areas. 3rd Series. A Series of Modern
  Surveys in Mathematics]}.
\newblock Springer-Verlag, Berlin, second edition, 1998.

\bibitem[G\"1]{G}
L.~G\"{o}ttsche.
\newblock Hilbert schemes of points on surfaces.
\newblock In {\em Proceedings of the {I}nternational {C}ongress of
  {M}athematicians, {V}ol. {II} ({B}eijing, 2002)}, pages 483--494. Higher Ed.
  Press, Beijing, 2002.

\bibitem[G\"2]{G2}
Lothar G\"{o}ttsche.
\newblock The {B}etti numbers of the {H}ilbert scheme of points on a smooth
  projective surface.
\newblock {\em Math. Ann.}, 286(1-3):193--207, 1990.

\bibitem[Gro1]{Gr}
I.~Grojnowski.
\newblock Instantons and affine algebras. {I}. {T}he {H}ilbert scheme and
  vertex operators.
\newblock {\em Math. Res. Lett.}, 3(2):275--291, 1996.

\bibitem[Gro2]{Gro}
Alexander Grothendieck.
\newblock Techniques de construction et th\'{e}or\`emes d'existence en
  g\'{e}om\'{e}trie alg\'{e}brique. {IV}. {L}es sch\'{e}mas de {H}ilbert.
\newblock In {\em S\'{e}minaire {B}ourbaki, {V}ol. 6}, pages Exp. No. 221,
  249--276. Soc. Math. France, Paris, 1995.

\bibitem[Hil]{H}
David Hilbert.
\newblock {\em Gesammelte {A}bhandlungen. {B}and {II}: {A}lgebra,
  {I}nvariantentheorie, {G}eometrie}.
\newblock Zweite Auflage. Springer-Verlag, Berlin-New York, 1970.

\bibitem[LQW]{LQW}
Wei-ping Li, Zhenbo Qin, and Weiqiang Wang.
\newblock Vertex algebras and the cohomology ring structure of {H}ilbert
  schemes of points on surfaces.
\newblock {\em Math. Ann.}, 324(1):105--133, 2002.

\bibitem[LS]{LS}
Manfred Lehn and Christoph Sorger.
\newblock Symmetric groups and the cup product on the cohomology of {H}ilbert
  schemes.
\newblock {\em Duke Math. J.}, 110(2):345--357, 2001.

\bibitem[Mac]{M}
F.~S. MacAulay.
\newblock Some {P}roperties of {E}numeration in the {T}heory of {M}odular
  {S}ystems.
\newblock {\em Proc. London Math. Soc. (2)}, 26:531--555, 1927.

\bibitem[Nak1]{N}
Hiraku Nakajima.
\newblock {\em Lectures on {H}ilbert schemes of points on surfaces}, volume~18
  of {\em University Lecture Series}.
\newblock American Mathematical Society, Providence, RI, 1999.

\bibitem[Nak2]{N2}
Hiraku Nakajima.
\newblock More lectures on {H}ilbert schemes of points on surfaces.
\newblock In {\em Development of moduli theory---{K}yoto 2013}, volume~69 of
  {\em Adv. Stud. Pure Math.}, pages 173--205. Math. Soc. Japan, [Tokyo], 2016.

\bibitem[SS]{SS}
Roy Skjelnes and Gregory~G. Smith.
\newblock Smooth {H}ilbert schemes: their classification and geometry.
\newblock {\em preprint}, 2020.
\newblock https://arxiv.org/abs/2008.08938.

\end{thebibliography}
\end{document}